\documentclass[11pt]{article}

\usepackage[T2A]{fontenc}
\usepackage[cp1251]{inputenc}
\usepackage{amsthm}
\usepackage{amsfonts}
\usepackage{eufrak}
\usepackage{amssymb}
\usepackage{amsmath}
\usepackage{cite}

 \usepackage{amsthm}
 \usepackage{bm}
\begin{document}
\newtheoremstyle{mytheorem}
  {\topsep}   
  {\topsep}   
  {\itshape}  
  {}       
  {\bfseries} 
  {  }         
  {5pt plus 1pt minus 1pt} 
  { }          
\newtheoremstyle{myremark}
  {\topsep}   
  {\topsep}   
  {\upshape}  
  {}       
  {\bfseries} 
  {  }         
  {5pt plus 1pt minus 1pt} 
  { }          
\theoremstyle{mytheorem}
\newtheorem{theorem}{Theorem}[section]
 \newtheorem{theorema}{Theorem}
\newtheorem{proposition}[theorem]{Proposition} 
\newtheorem{lemma}[theorem]{Lemma} 
\newtheorem{corollary}[theorem]{Corollary} 
\newtheorem{definition}[theorem]{Definition} 
\theoremstyle{myremark}
\newtheorem{remark}[theorem]{Remark} 
\noindent This article has been accepted for publication in\\
Journal of Mathematical Physics, Analysis, Geometry

\vskip 1 cm

\noindent{\textbf{\Large Arithmetic  of a certain semigroup of probability}}

\medskip

\noindent{\textbf{\Large distributions on the group $\mathbb{R}\times \mathbb{Z}(2)$}}

\bigskip

\noindent{\textbf{Gennadiy Feldman}}



\bigskip

\noindent{\textbf{Abstract}}

\bigskip

\noindent We consider a certain convolution semigroup $\Theta$ of probability distributions 
on the group 
$\mathbb{R}\times \mathbb{Z}(2)$, where $\mathbb{R}$ is the group 
of real numbers  and  $\mathbb{Z}(2)$ is the additive group of the 
integers modulo 2. This semigroup appeared  in connection with the study of a characterization problem of mathematical statistics on ${\bm a}$-adic solenoids containing an element of order 2. 
We   answer the  questions that arise 
in the study of arithmetic of the semigroup $\Theta$. 
Namely, we describe the class of infinitely divisible distributions, 
the class  of indecomposable distributions, and the class  of 
distributions which  have no indecomposable factors.


\bigskip


\section{Introduction}

A number of works are devoted to arithmetic of various semigroups of probability distributions (see e.g. \cite{O1}, \cite{Tr1}, \cite{Tr2}). The purpose of this note is to study the arithmetic of a certain semigroup  of 
probability distributions on the direct product of the group of real numbers and 
the additive group of the integers modulo  2. 
This semigroup appears  in connection with the study of a characterization problem of mathematical statistics on ${\bm a}$-adic solenoids containing an element of order 2
(\!\!\cite{F_solenoid}, see also   \cite[\S 11]{book2023}).

Let $X$ be a locally compact Abelian group. Denote by $Y$  the character group of the group $X$ and by
 $(x,y)$ the value of a character $y \in Y$ at an element $x
\in X$. Denote by $\mathrm{M}^1(X)$ the convolution semigroup of all 
distributions (probability measures)
on the group $X$. Let  $\mu\in\mathrm{M}^1(X)$. Denote by
$$
\hat\mu(y) =
\int_{X}(x, y)d \mu(x), \quad y\in Y,$$   the characteristic function of 
the distribution $\mu$. The characteristic function of a signed measure on 
the group $X$ is defined   in the same way.  Denote by  $m_K$ the 
Haar distribution on a compact subgroup   $K$ of the group $X$.

Recall the following definitions. 
Let  $\mu\in\mathrm{M}^1(X)$. A distribution   $\mu_1\in\mathrm{M}^1(X)$ 
is called a \textit{factor} of $\mu$  if there is a distribution   
$\mu_2\in\mathrm{M}^1(X)$ such that the equality 
\begin{equation}\label{08.11.1}
\mu=\mu_1*\mu_2 
\end{equation}
holds. A distribution with support only at a single point $x\in X$ is called 
  \textit{degenerate} and is denoted by  $E_x$.
A nondegenerate distribution $\mu\in\mathrm{M}^1(X)$ is called \textit{indecomposable} 
if it has only degenerate distributions or shifts $\mu$ as factors.
A distribution $\mu\in\mathrm{M}^1(X)$ is called \textit{decomposable} if there are nondegenerate distributions   $\mu_1$ and $\mu_2$ such that (\ref{08.11.1}) holds. 
A distribution $\mu\in\mathrm{M}^1(X)$ 
is said to be \textit{infinitely divisible} if, for each natural $n$, there 
are a distribution $\mu_n\in\mathrm{M}^1(X)$ and an element $x_n\in X$ such 
that $\mu=\mu_n^{*n}*E_{x_n}$. We note that this definition is slightly 
different from the 
classical one in the case of the   group of real numbers. The shift by the element $x_n$ is necessary, in particular, for all degenerate distributions to be infinitely divisible. 

 Denote by $\mathbb{R}$ the group of real numbers  and by $\mathbb{Z}(2)=\{0, 1\}$ the additive group of the integers modulo  2. Consider the group $\mathbb{R}\times \mathbb{Z}(2)$. Denote by  $(t, k)$, where $t\in \mathbb{R}$, $k\in \mathbb{Z}(2)$, its elements. The character group of the group $\mathbb{R}\times \mathbb{Z}(2)$ is topologically isomorphic to the group $\mathbb{R}\times \mathbb{Z}(2)$.   
 Denote by $(s, l)$, $s\in \mathbb{R}$,
 $l\in \mathbb{Z}(2)$,   elements of the character group of the group $\mathbb{R}\times \mathbb{Z}(2)$. The value of a character   $(s, l)$ at an 
 element $(t, k)\in \mathbb{R}\times \mathbb{Z}(2)$ is defined by the formula
 $$((t, k), (s, l))=e^{its}(-1)^{kl}.$$ 

 Let $\mu\in\mathrm{M}^1(\mathbb{R}\times \mathbb{Z}(2))$ and assume that the support 
 of  $\mu$ is contained in the subgroup   $\mathbb{Z}(2)$, i.e., 
 $\mu\{(0, 0)\}=a\ge 0$, $\mu\{(0, 1)\}=b\ge 0$, where $a+b=1$. 
 Then the characteristic function  $\hat\mu(s, l)$ is of the form
\begin{equation}\label{14.12.1}
\hat\mu(s, l)= \begin{cases}1,
&\text{\ if\ }\  s\in \mathbb{R}, \    l=0,\\ \kappa,
& \text{\ if\ }\ s\in \mathbb{R}, \     l=1,
\end{cases}
\end{equation}
where $\kappa=a-b$. In particular, the characteristic function of the Haar distribution   $m_{\mathbb{Z}(2)}$,  is of the form
\begin{equation}\label{7.1}
\widehat m_{\mathbb{Z}(2)}(s, l)= \begin{cases}1,
&\text{\ if\ }\  s\in \mathbb{R}, \    l=0,\\ 0,
& \text{\ if\ }\ s\in \mathbb{R}, \     l=1.
\end{cases}
\end{equation}
Denote by $\Gamma(\mathbb{R})$ the set of Gaussian distributions on the group $\mathbb{R}$.

\section{Class $\Theta$}

  Let $\mu$ be a distribution on the group  
  $\mathbb{R}\times \mathbb{Z}(2)$ such that $\mu\in \Gamma(\mathbb{R})*\mathrm{M}^1(\mathbb{Z}(2))$, i.e., $\mu=\gamma*\omega$,
where $\gamma\in\Gamma(\mathbb{R})$, $\omega\in\mathrm{M}^1(\mathbb{Z}(2))$, 
and the groups $\mathbb{R}$ and  $\mathbb{Z}(2)$ are considered as subgroups 
of the group $\mathbb{R}\times \mathbb{Z}(2)$.
It is easy to see that the characteristic function 
   $\hat\mu(s, l)$ is of the form
$$
\hat\mu(s, l) = \begin{cases}\exp\{-\sigma s^2+i\beta s\},
&\text{\ if\ }\  s\in \mathbb{R}, \    l=0,\\ \kappa\exp\{-\sigma s^2+i\beta s\},
& \text{\ if\ }\ s\in \mathbb{R}, \     l=1,
\\
\end{cases}
$$
where $\sigma \ge 0$, $\beta$ and $\kappa$ are real numbers, $|\kappa|\le 1$.
Let us introduce a class of distributions on the group
$\mathbb{R}\times \mathbb{Z}(2)$  which is much broader than the class $\Gamma(\mathbb{R})*\mathrm{M}^1(\mathbb{Z}(2))$.
For this purpose we need the following assertion proved in \cite{F_solenoid}, see also   \cite[Lemma 11.1]{book2023}.
For the sake of completeness, we present here its proof.

\begin{lemma}\label{lem6} Consider the group   $\mathbb{R}\times \mathbb{Z}(2)$.  Let $f(s, l)$ be a function on the character group of the group 
$\mathbb{R}\times \mathbb{Z}(2)$ of the form 
\begin{equation}\label{y6}
f(s, l) = \begin{cases}\exp\{-\sigma s^2+i\beta s\}, &\text{\ if\ }\  s\in \mathbb{R}, \ l=0,\\ \kappa\exp\{-\sigma' s^2+i\beta's\}, & \text{\ if\ }\ s\in \mathbb{R}, \   l=1,
\\
\end{cases}
\end{equation}
where $\sigma\ge 0$,  $\sigma'\ge 0$ and  $\beta$, $\beta'$, 
$\kappa$ are real numbers. Then
$f(s, l)$ is the characteristic function of a signed measure   $\mu$ on the group $\mathbb{R}\times \mathbb{Z}(2)$.
The signed measure $\mu$ is a measure if and only if either
\begin{equation}\label{15.3}
0<\sigma'<\sigma,  \quad  0<|\kappa|\le\sqrt\frac{\sigma'}{\sigma}
\exp\left\{-\frac{(\beta-\beta')^2}{4(\sigma-\sigma')}\right\},\end{equation}
or
\begin{equation}\label{15.4}
\sigma=\sigma', \quad\beta=\beta', \quad  |\kappa|\le 1.
\end{equation}
Moreover, if    $(\ref{15.4})$ is fulfilled, then 
$\mu\in \Gamma(\mathbb{R})*\mathrm{M}^1(\mathbb{Z}(2))$.
\end{lemma}

\begin{proof} Let $\kappa=0$. Then $f(s, l)$ is the characteristic 
function of the distribution $\mu$ of the form 
$\mu=\gamma*m_{\mathbb{Z}(2)}$, where
 $\gamma\in\Gamma(\mathbb{R})$. Therefore, we can assume
    that $\kappa\ne 0$. Multiplying, if necessary, the function $f(s, l)$ by a suitable character of the
group $\mathbb{Z}(2)$, we can suppose, without loss of generality, that $ \kappa>0$.
Take a number  $a> 0$  and denote by $\gamma_a$ a Gaussian distribution on the group 
$\mathbb{R}$ with 
the density
\begin{equation}\label{y5}
\rho_a(t)=\frac{1}{2\sqrt{\pi a}}\exp\left\{-\frac{t^2}{4a}\right\},
\quad t\in \mathbb{R}.
\end{equation}
It is obvious that
$$
\hat\gamma_a(s)=\exp\{-a s^2\}, \quad s\in \mathbb{R}. 
$$
 
Let $\mu$ be the signed measure on the group  $\mathbb{R}\times \mathbb{Z}(2)$  
which is defined by the following way 
$$
\mu(B\times \{k\})= 
\begin{cases}
{\frac{1}{2}(\gamma_{\sigma}*E_{\beta}
+ \kappa\gamma_{\sigma'}*E_{\beta'})}(B), & \text{\ if\ }\   k=0,
\\  \frac{1}{2}(\gamma_{\sigma}*E_{\beta}
- \kappa\gamma_{\sigma'}*E_{\beta'})(B), & \text{\ if\ }\ k=1,
\end{cases}
$$
where $B$ is a Borel subset of   $\mathbb{R}$.
Put
$$\lambda_0={\frac{1}{2}(\gamma_{\sigma}*E_{\beta}
+ \kappa\gamma_{\sigma'}*E_{\beta'})}, \quad {\lambda_1=\frac{1}{2}(\gamma_{\sigma}*E_{\beta}
- \kappa\gamma_{\sigma'}*E_{\beta'})}.$$
Taking into account that $$\hat\lambda_0(s)+
\hat\lambda_1(s)=\hat\gamma_\sigma(s)e^{i\beta s}$$ and $$\hat\lambda_0(s)-
\hat\lambda_1(s)= \kappa\hat\gamma_{\sigma'}(s)e^{i\beta's},$$  we have
\begin{multline*}
\hat\mu(s, l)=\int\limits_{\mathbb{R}\times
\mathbb{Z}(2)}e^{its}(-1)^{kl} d\mu(t, k)=\int\limits_{\mathbb{R}\times \{0\}}e^{its}d\mu(t, 0)\\+\int\limits_{\mathbb{R}\times \{1\}}e^{its}(-1)^{l}d\mu(t, 1)=f(s, l).
\end{multline*}
Thus, $f(s, l)$ is the characteristic function of the signed measure   $\mu$. 
Moreover, the signed measure $\mu$ is a measure if and only if 
the signed measure  
$\lambda_1$ is a measure. It is obvious that if the signed measure  
$\lambda_1$ is a measure, then either   $\sigma > 0$ and  $\sigma'> 0$ or $\sigma=\sigma'= 0$.
It is clear that if    $\sigma=\sigma'= 0$, then the signed measure
$\mu$ is a measure if and only if   $\beta=\beta'$ and  $ \kappa\le 1$.
In this case the lemma is proved.

Let $\sigma > 0$ and  $\sigma'> 0$.
In view of (\ref{y5}),  the signed measure
$\lambda_1$ is a measure if and only if  the equality
 $$
\frac{1}{2\sqrt{\pi \sigma}}\exp\left\{-\frac{(t-\beta)^2}{4\sigma}\right\}
-\frac{\kappa}{2\sqrt{\pi \sigma'}}
\exp\left\{-\frac{(t-\beta')^2}{4\sigma'}\right\}\ge 0 
$$
holds for all $t\in \mathbb{R}$. This inequality is equivalent to the following
\begin{equation}\label{02_11_1}
\kappa\le\sqrt\frac{\sigma'}{\sigma}
\exp\left\{-\frac{(t-\beta)^2}{4\sigma}
+\frac{(t-\beta')^2}{4\sigma'}\right\}, \quad t\in \mathbb{R}.
\end{equation}

Suppose that $\sigma=\sigma'$. Then it follows from (\ref{02_11_1}) that $\beta=\beta'$ and  
$\kappa\le 1$.

Let $\sigma\ne\sigma'$. Inasmuch as $\kappa>0$, we have $\sigma'<\sigma$. 
The minimum of 
the function on the right side of   inequality (\ref{02_11_1}) is reached at the point 
$$t_0=\frac{\sigma\beta'-\sigma'\beta}{\sigma-\sigma'},$$ and it is equal to 
$$
\sqrt\frac{\sigma'}{\sigma}
\exp\left\{-\frac{(\beta-\beta')^2}{4(\sigma-\sigma')}\right\}.
$$
It follows from the above that the signed measure
$\lambda_1$, and hence the signed measure  $\mu$ is a measure if and only if 
either  (\ref{15.3}) or (\ref{15.4}) is fulfilled.
 It is also obvious that if (\ref{15.4}) holds, then
  $\mu\in \Gamma(\mathbb{R})*\mathrm{M}^1(\mathbb{Z}(2))$.   
  \end{proof}
  
\begin{definition} \label{de1}  We say that a distribution   $\mu$ on the group $\mathbb{R}\times \mathbb{Z}(2)$ belongs to the class $\Theta$  if     $\hat\mu(s, l)=f(s, l)$, where the function  $f(s, l)$ is represented in the form $(\ref{y6})$ 
 and either $(\ref{15.3})$ or $(\ref{15.4})$ holds.
\end{definition}

Since the product of characteristic functions corresponds to the   convolution of distributions,  it follows from the Lemma \ref{lem6} that the class $\Theta$ is a convolution semigroup. 
The purpose of this note is to answer the main questions that arise 
in the study of the arithmetic of the semigroup $\Theta$. 
 Namely,   we describe in $\Theta$ the class of infinitely divisible distributions, 
the class  of indecomposable distributions, and the class  of 
distributions which  have no indecomposable factors.
 
 \begin{remark}\label{re1}
Heyde's theorem on characterization of the Gaussian distribution 
on the real line by the symmetry of the conditional distribution of one 
linear form of independent random variables given another   is well known 
(\!\!\cite[\S\,13.4.1]{KaLiRa}). The class of 
distributions $\Theta$  
arises in connection with the study of an analogue of this theorem
 for the group $\mathbb{R}\times \mathbb{Z}(2)$. 
 
  Let $a$ be a topological automorphism 
of the group $\mathbb{R}\times \mathbb{Z}(2)$. 
It is obvious that
 $a$ is of the form
 $a(t, k)=(c_a t, k)$, where $c_a\in \mathbb{R}$, $c_a\ne 0$.
 We   identify $a$ and $c_a$, i.e., we   write
$a(t, k)=(a t, k)$ and assume that
$a\in \mathbb{R}$, $a\ne 0$.
The following group analogue of Heyde's  theorem for the group $\mathbb{R}\times \mathbb{Z}(2)$ 
was proved in \cite{POTA}, see also \cite[Theorem 11.6]{book2023}).

\textit{Consider the group $\mathbb{R}\times \mathbb{Z}(2)$  and 
let $a_j$, $b_j$, $j = 1, 2,\dots, n$, $n \ge 2,$ be topological automorphisms of 
$\mathbb{R}\times \mathbb{Z}(2)$
satisfying the conditions $b_ia_i^{-1} + b_ja_j^{-1}\ne 0$  for all $i, j$. Let
$\xi_j$ be independent random variables with values in the group    
$\mathbb{R}\times \mathbb{Z}(2)$ and distributions $\mu_j$ with nonvanishing characteristic functions.
If the conditional distribution of the linear form
$L_2 = b_1\xi_1 + \dots + b_n\xi_n$ given $L_1 = a_1\xi_1 + \dots + a_n\xi_n$ is symmetric, 
 then all distributions $\mu_{j}$ belong to the class $\Theta$.}

The class of distributions $\Theta$  also arises   
 in connection with the study of an analogue of Heyde's   theorem
 on ${\bm a}$-adic solenoids containing an element of order 2 
 (\!\!\cite{F_solenoid}, see 
 also \cite[Theorem 11.20]{book2023}). 
 Note also that some problems related to independent random variables 
 with values in the group
  $\mathbb{R}\times \mathbb{Z}(2)$ 
   were studied in  \cite{Il}, \cite{Tr3}, and \cite{Zo}.
   \end{remark}

\section{Arithmetic of  the semigroup $\Theta$}
 
The proof of the main theorem is based on the following lemma.

\begin{lemma}\label{lem14}      Let    $\mu\in\Theta$ 
and   $\hat\mu(s, l)=f(s, l)$, where the function $f(s, l)$ 
is represented in the form $(\ref{y6})$ and
\begin{equation}\label{11_12_7}
0<\sigma'<\sigma, \quad |\kappa|=
\sqrt\frac{\sigma'}{\sigma}\exp\left\{-\frac{(\beta-\beta')^2}{4(\sigma-\sigma')}\right\} 
\end{equation}
is fulfilled. Then   $\mu$ is an indecomposable distribution.
\end{lemma}
\begin{proof} We break the proof into several steps.

1. Assume $\mu=\mu_1*\mu_2$, where 
$\mu_j\in \mathrm{M}^1(\mathbb{R}\times \mathbb{Z}(2))$  and 
$\mu_j$ are nondegenerate  distributions. We have
\begin{equation}\label{14.05.1}
\hat\mu(s, l)=\hat\mu_1(s, l)\hat\mu_2(s, l), \quad 
s\in \mathbb{R}, \ l\in \mathbb{Z}(2).
\end{equation}
Substituting $l=0$ in (\ref{14.05.1}), we obtain
$$
\exp\{-\sigma s^2+i\beta s\}=\hat\mu_1(s, 0)\hat\mu_2(s, 0), \quad 
s\in \mathbb{R}.
$$
By Cram\'er's theorem on decomposition of the Gaussian distribution 
on the real line,   
$$
\hat\mu_j(s, 0)=\exp\{-\sigma_j s^2+i\beta_j s\}, \quad 
s\in \mathbb{R},
$$
where $\sigma_j\ge 0$, $\beta_j\in \mathbb{R}$, $j=1, 2$. It follows from 
definition of the characteristic function that $\hat\mu_j(s, 1)$ is an 
entire function and
\begin{equation}\label{14.05.2}
\max_{s\in \mathbb{C}, \ |s|\le r}|\hat\mu_j(s, 1)|\le
\max_{s\in \mathbb{C}, \ |s|\le r}|\exp\{-\sigma_j s^2+i\beta_j s\}|. 
\end{equation}
It follows from (\ref{14.05.1}) that the   functions $\hat\mu_j(s, 1)$ do 
not vanish in the complex plane $\mathbb{C}$. In view of (\ref{14.05.2}), 
the entire functions
$\hat \mu_j(s, 1)$ are of at most order 2 and 
type $\sigma_j$. Taking into account that 
$\hat\mu_j(-s, 1)=\overline{\hat \mu_j(s, 1)}$, 
by the Hadamard theorem on the representation 
of an entire function
 of finite order and Lemma \ref{lem6}, 
 we obtain  that  the characteristic 
functions   $\hat \mu_j(s, l)$ can be written in the form
\begin{equation}\label{14.05.3}
\hat\mu_j(s, l)=\begin{cases}\exp\{-\sigma_j s^2+i\beta_j s\}, &\text{\ if\ }\ s\in \mathbb{R}, \ l=0,\\ \kappa_j\exp\{-\sigma_j' s^2+i\beta_j's\},
&\text{\  if\ }\ s\in \mathbb{R}, \   l=1,
\\
\end{cases}
\end{equation}
where either   
$$
0<\sigma_j'<\sigma_j,  \quad  0<|\kappa_j|\le\sqrt\frac{\sigma_j'}
{\sigma_j}\exp\left\{-\frac{(\beta_j-\beta_j')^2}{4(\sigma_j-\sigma_j')}\right\}, \  j=1, 2,
$$
or
$$
0\le \sigma_j=\sigma'_j, \quad \beta_j=\beta_j', \quad    0<|\kappa_j|\le 1, \  j=1, 2.
$$
Moreover, $\mu_j\in\Theta$. 

Note that in this discussion,  we used only  that 
$\mu\in\Theta$, i.e., $\hat\mu(s, l)=f(s, l)$, where the function  $f(s, l)$ is represented in the form $(\ref{y6})$ 
 and either $(\ref{15.3})$ or $(\ref{15.4})$ holds 
 and the fact that the characteristic function  $\hat\mu(s, l)$ does not vanish,
 i.e., in  (\ref{y6})  $\kappa\ne 0$.

2. It follows from (\ref{y6}) and (\ref{14.05.3}) that
\begin{equation}\label{14.12.10}
\kappa=\kappa_1\kappa_2,\quad \sigma=\sigma_1+\sigma_2,\quad \sigma'=\sigma'_1+\sigma'_2,\quad \beta=\beta_1+\beta_2,\quad \beta'=\beta'_1+\beta'_2.
\end{equation}
 Since
$0<\sigma'<\sigma$, it follows from (\ref{14.12.10}) that
for at least one $j$, say for $j=1$,  the inequality $\sigma_1'<\sigma_1$ holds, 
i.e., $\mu_1\notin\Gamma(\mathbb{R})*\mathrm{M}^1(\mathbb{Z}(2))$. Hence  the inequalities
\begin{equation}\label{31.10}
0<\sigma_1'<\sigma_1, \quad  0<|\kappa_1|\le\sqrt\frac{\sigma_1'}
{\sigma_1}\exp\left\{-\frac{(\beta_1-\beta_1')^2}{4(\sigma_1-\sigma_1')}\right\} 
\end{equation}
are fulfilled.

There are two possibilities for the distribution   $\mu_2$: either 
$\mu_2\in\Gamma(\mathbb{R})*\mathrm{M}^1(\mathbb{Z}(2))$ or 
$\mu_2\notin\Gamma(\mathbb{R})*\mathrm{M}^1(\mathbb{Z}(2))$.

3. Let $\mu_2\in\Gamma(\mathbb{R})*\mathrm{M}^1(\mathbb{Z}(2))$. Then we  have
 $0<\sigma_2=\sigma'_2$, $\beta_2=\beta'_2$  and  $0<|\kappa_2|\le 1$. Moreover,
the equality $\kappa=\kappa_1\kappa_2$ implies that $|\kappa|\le |\kappa_1|$. In view of 
  (\ref{11_12_7})  and (\ref{31.10}), it follows from this that
\begin{multline*}|\kappa|=\sqrt\frac{\sigma'}
{\sigma}\exp\left\{-\frac{(\beta-\beta')^2}{4(\sigma-\sigma')}\right\}
=\sqrt\frac{\sigma_1'+\sigma_2}{\sigma_1+\sigma_2}
\exp\left\{-\frac{(\beta_1-\beta_1')^2}{4(\sigma_1-\sigma_1')}\right\}\\ 
\le |\kappa_1|\le \sqrt\frac{\sigma_1'}{\sigma_1}
\exp\left\{-\frac{(\beta_1-\beta_1')^2}{4(\sigma_1-\sigma_1')}\right\}.
\end{multline*}
Hence,
$$
\sqrt\frac{\sigma_1'+\sigma_2}{\sigma_1+\sigma_2}
\le\sqrt\frac{\sigma_1'}{\sigma_1},
$$
which is obviously impossible because $0<\sigma_1'<\sigma_1$ and $\sigma_2>0$.

4. Let $\mu_2\notin\Gamma(\mathbb{R})*\mathrm{M}^1(\mathbb{Z}(2))$. Then the inequalities
\begin{equation}\label{31.11}
0<\sigma_2'<\sigma_2, \quad  0<|\kappa_2|\le\sqrt\frac{\sigma_2'}
{\sigma_2}\exp\left\{-\frac{(\beta_2-\beta_2')^2}{4(\sigma_2-\sigma_2')}\right\} 
\end{equation}
hold. Taking into account  (\ref{14.12.10})--(\ref{31.11}),   we obtain
\begin{multline*}
|\kappa|=\sqrt\frac{\sigma'}{\sigma}
\exp\left\{-\frac{(\beta-\beta')^2}{4(\sigma-\sigma')}\right\}
=\sqrt\frac{\sigma_1'+\sigma'_2}{\sigma_1+\sigma_2}
\exp\left\{-\frac{(\beta_1+\beta_2-\beta_1'-\beta_2')^2}
{4(\sigma_1+\sigma_2-\sigma_1'-\sigma_2')}\right\}\\=|\kappa_1\kappa_2|\le \sqrt\frac{\sigma_1'}{\sigma_1}
\exp\left\{-\frac{(\beta_1-\beta_1')^2}{4(\sigma_1-\sigma_1')}\right\}\sqrt\frac{\sigma'_2}{\sigma_2}
\exp\left\{-\frac{(\beta_2-\beta_2')^2}{4(\sigma_2-\sigma_2')}\right\}.
\end{multline*}
Hence,
\begin{multline}\label{12.12.1}
\sqrt\frac{\sigma_1'+\sigma'_2}{\sigma_1+\sigma_2}
\exp\left\{-\frac{(\beta_1+\beta_2-\beta_1'-\beta_2')^2}
{4(\sigma_1+\sigma_2-\sigma_1'-\sigma_2')}\right\}\\ 
\le \sqrt\frac{\sigma_1'}{\sigma_1}
\exp\left\{-\frac{(\beta_1-\beta_1')^2}{4(\sigma_1-\sigma_1')}\right\}\sqrt\frac{\sigma'_2}{\sigma_2}\exp\left\{-\frac{(\beta_2-\beta_2')^2}
{4(\sigma_2-\sigma_2')}\right\}.
\end{multline}

It is easy to see that the inequalities  $0<\sigma'_1<\sigma_1$ and
  $0<\sigma'_2<\sigma_2$ imply the inequality
\begin{equation}\label{31.4}
\sqrt\frac{\sigma_1'}{\sigma_1}\sqrt\frac{\sigma'_2}{\sigma_2}<
\sqrt\frac{\sigma_1'+\sigma'_2}{\sigma_1+\sigma_2}.
\end{equation}

Note that if $a, b\in \mathbb{R}$, $c>0$, $d>0$, then the inequality
\begin{equation}\label{31.5}
\frac{(a+b)^2}{c+d}\le \frac{a^2}{c}+\frac{b^2}{d} 
\end{equation}
is fulfilled.
Substituting $a=\beta_1-\beta_1'$, $b=\beta_2-\beta_2'$, $c=\sigma_1-\sigma_1'$,
$d=\sigma_2-\sigma_2'$ in (\ref{31.5}), we get from the obtained inequality 
\begin{multline}\label{31.6}
\exp\left\{-\frac{(\beta_1-\beta_1')^2}{4(\sigma_1-\sigma_1')}\right\}
\exp\left\{-\frac{(\beta_2-\beta_2')^2}{4(\sigma_2-\sigma_2')}\right\}\\ 
\le\exp\left\{-\frac{(\beta_1+\beta_2-\beta_1'-\beta_2')^2}
{4(\sigma_1+\sigma_2-\sigma_1'-\sigma_2')}\right\}.
\end{multline}
Note that the inequality (\ref{12.12.1}) contradicts the inequality that results when we multiply (\ref{31.4}) and (\ref{31.6}).

  We   assumed that both distributions   $\mu_1$ and $\mu_2$ are nondegenerate 
  and   came to a contradiction. Thus, at least one of the distributions $\mu_j$ is degenerate, i.e., $\mu$ is an indecomposable distribution.  
\end{proof}
\begin{remark}
Consider the group $\mathbb{R}\times \mathbb{Z}(2)$  and let $\mu\in\Theta$. 
If $\mu$ is an infinitely divisible distribution in the semigroup 
$\mathrm{M}^1(\mathbb{R}\times \mathbb{Z}(2))$, then $\mu$ is an 
infinitely divisible distribution in the semigroup $\Theta$, i.e.,   
distributions $\mu_n$ in the definition of an infinitely divisible 
distribution belong to the semigroup $\Theta$. 
Indeed, if the characteristic function of $\mu$ does not vanish, it
follows from the proof of item 1 of Lemma \ref{lem14}. 
If the characteristic function of $\mu$  vanishes, then $\mu$
can be represented in the form $\mu=\gamma*m_{\mathbb{Z}(2)}$,
where $\gamma\in\Gamma(\mathbb{R})$. In this case, the statement 
is obviously true. 
   
Note that if the characteristic function of $\mu$  vanishes, then $\mu$
is decomposable in the semigroup 
$\Theta$. Hence if $\mu$ is indecomposable in $\Theta$, then
the characteristic function of $\mu$ does not vanish.
Then it follows from the proof of item 1 of Lemma \ref{lem14} 
that $\mu$ is also indecomposable in the semigroup 
$\mathrm{M}^1(\mathbb{R}\times \mathbb{Z}(2))$.
\end{remark}

As proven in \cite[Chapter IV, Theorem 11.3]{Pa}, the following assertion holds.

\medskip

\textit{Let $X$ be a second countable locally compact Abelian group and $\mu\in\mathrm{M}^1(X)$. 
Assume that $\mu$ has no factors of the form $m_K$, where $K$ is a nonzero compact subgroup of the group $X$. Then the distribution $\mu$ can be represented as a convolution of a finite or countable number of indecomposable distributions and a distribution that   has no  indecomposable factors.}

\medskip

Using Lemma \ref{lem14}, for distributions belonging to the semigroup $\Theta$  this assertion can be considerably strengthened. Note that if $\mu\in\Theta$ and the Haar distribution $m_{\mathbb{Z}(2)}$ is not a factor of $\mu$, then the characteristic function $\hat\mu(s, l)$ does not vanish. Note also that by the classical Cram\'er theorem, Gaussian distributions on the group $\mathbb{R}$   have no   indecomposable factors.

 \begin{proposition}\label{pr3}    Let $\mu\in\Theta$ and
   $\mu$ be a nondegenerate distribution. Then either 
 $\mu\in\Gamma(\mathbb{R})*\mathrm{M}^1(\mathbb{Z}(2))$,  
 or $\mu$ is an indecomposable distribution,  or $\mu =\nu*\gamma$, where 
 $\nu\in\Theta$, 
 $\nu$ is an   indecomposable distribution, and $\gamma$ is a 
 nondegenerate Gaussian distribution 
 on the group $\mathbb{R}$.
\end{proposition}

\begin{proof} Let $\mu\notin\Gamma(\mathbb{R})*\mathrm{M}^1(\mathbb{Z}(2))$. Then 
$\hat\mu(s, l)=f(s, l)$, where the function $f(s, l)$ is represented in the form $(\ref{y6})$ and inequalities $(\ref{15.3})$ hold.
If    (\ref{11_12_7}) is fulfilled, then by Lemma \ref{lem14},   $\mu$ is an indecomposable distribution. Assume that (\ref{11_12_7}) is  false. Then the inequalities 
\begin{equation}\label{14_12_2}
0<|\kappa|<\sqrt\frac{\sigma'}{\sigma}
\exp\left\{-\frac{(\beta-\beta')^2}{4(\sigma-\sigma')}\right\} 
\end{equation}
hold. Put
\begin{equation}\label{11_12_6}
b=\exp\left\{-\frac{(\beta-\beta')^2}{4(\sigma-\sigma')}\right\}.
\end{equation}
We have
\begin{equation}\label{11_12_1}
0<|\kappa|<\sqrt\frac{\sigma'}{\sigma}b.
\end{equation}
Put \begin{equation}\label{11_12_2}
a=\frac{{\sigma\kappa^2-\sigma'b^2}}{{\kappa^2-b^2}}.
\end{equation}
It follows from  (\ref{11_12_1}) that $0<a<\sigma'$, and  (\ref{11_12_2}) 
implies that
\begin{equation}\label{11_12_3}
|\kappa|=\sqrt\frac{{\sigma'-a}}{{\sigma-a}}b.
\end{equation}
By Lemma \ref{lem6},   we get from (\ref{11_12_6}) and (\ref{11_12_3}) that there 
is a distribution  $\nu\in\Theta$  with the characteristic function of the form
\begin{equation}\label{11_12_4}
\hat \nu(s, l) = \begin{cases}\exp\{-(\sigma-a) s^2+i\beta s\}, &\text{\ if\ }\ s\in \mathbb{R}, \ l=0,\\ \kappa\exp\{-(\sigma'-a) s^2+i\beta's\}, &\text{\ if\ }\ s\in \mathbb{R}, \   l=1.
\\
\end{cases}
\end{equation}
In view of (\ref{11_12_6}) and (\ref{11_12_3}), by Lemma \ref{lem14},   $\nu$ is an indecomposable distribution.

Denote by $\gamma$ the Gaussian distribution on the group $\mathbb{R}$ with the characteristic function ${\hat\gamma(s)=\exp\{-as^2\}}$. If we consider $\gamma$ as a distribution on the group $\mathbb{R}\times \mathbb{Z}(2)$, then 
\begin{equation}\label{11_12_5}
\hat \gamma(s, l) = \begin{cases}\exp\{-as^2\}, &\text{\ if\ }\ 
s\in \mathbb{R}, \ l=0,\\ \exp\{-as^2\}, &\text{\ if\ }\ s\in \mathbb{R}, \   l=1.
\\
\end{cases}
\end{equation}
In view of $\hat\mu(s, l)=f(s, l)$, where the function $f(s, l)$ is represented in the form $(\ref{y6})$, it follows from   (\ref{11_12_4}) and  (\ref{11_12_5}) that 
$$
\hat\mu(s, l)=\hat\nu(s, l)\hat\gamma(s, l), \quad s\in  \mathbb{R}, \ l\in \mathbb{Z}(2).
$$  Hence, $\mu=\nu*\gamma$.

It is easy to see that $\gamma$ is the maximal, in the natural sense, Gaussian factor of the   distribution $\mu$.  
\end{proof}
\begin{corollary}\label{co2}
 Let   $\mu\in\Theta$  and    
 ${\mu\notin \Gamma(\mathbb{R})*\mathrm{M}^1(\mathbb{Z}(2))}$.
Then the distribution $\mu$ has an indecomposable factor.
\end{corollary}
We  will complement Proposition \ref{pr3}   with the following assertion.
\begin{proposition}\label{pr4}     Let $\mu\in\Theta$, 
  $\mu\notin\Gamma(\mathbb{R})*\mathrm{M}^1(\mathbb{Z}(2))$,
  and   $\mu$ be
a decomposable distribution. 
Then for each natural $n$ there are indecomposable distributions 
$\mu_j\in\Theta$, $j=1, 2, \dots, n$, and a nondegenerate Gaussian distribution 
$\gamma_n$ on the group $\mathbb{R}$  such that
$$
\mu=\mu_1*\mu_2*\cdots*\mu_n*\gamma_n.
$$
\end{proposition}
\begin{proof} In view of Proposition \ref{pr3}, it suffices to prove that the distribution $\mu$ can be represented in the form   $\mu=\mu_1*\nu$, where $\mu_1, \nu\in\Theta$,  $\mu_1$ is an indecomposable distribution, 
and $\nu\notin \Gamma(\mathbb{R})*\mathrm{M}^1(\mathbb{Z}(2))$ and $\nu$ is a decomposable distribution.

The condition of the proposition implies that   $\hat\mu(s, l)=f(s, l)$, where the function $f(s, l)$ is represented in the form $(\ref{y6})$ and  inequalities    $(\ref{14_12_2})$ hold.
Take numbers  $\sigma_1$ and   $\sigma_1'$ such that $0<\sigma_1'<\sigma_1$. Put $\kappa_1=\sqrt\frac{\sigma_1'}{\sigma_1}$. By Lemma \ref{lem6},   there is a distribution   $\mu_1\in\Theta$ such that  its characteristic function is of the form 
$$
\hat\mu_1(s, l) = \begin{cases}\exp\{-\sigma_1 s^2\}, &\text{\ if\ }
\ s\in \mathbb{R}, \ l=0,\\ \kappa_1\exp\{-\sigma_1' s^2\}, 
&\text{\ if\ }\ s\in \mathbb{R}, \   l=1.
\\
\end{cases}
$$
By Lemma \ref{lem14}, $\mu_1$ is an indecomposable distribution.

Put $\tau=\sigma-\sigma_1$, $\tau'=\sigma'-\sigma'_1$, $\varsigma=\kappa/\kappa_1$. We can assume that $\sigma_1$ is arbitrarily small and $\kappa_1$ is arbitrarily close to 1. Then (\ref{14_12_2}) implies that the inequalities
\begin{equation}\label{eq21.06.1}
0<|\varsigma|<\sqrt\frac{\tau'}{\tau}
\exp\left\{-\frac{(\beta-\beta')^2}{4(\tau-\tau')}\right\} 
\end{equation}
are valid.
By Lemma   \ref{lem6},  it follows from this that
there is a distribution $\nu\in\Theta$ with the characteristic function
\begin{equation}\label{eq21.06.2}
\hat\nu(s, l) = \begin{cases}\exp\{-\tau s^2+i\beta s\}, &\text{\ if\ }\ s\in \mathbb{R}, \ l=0,\\  \varsigma\exp\{-\tau'  s^2+i\beta' s\}, &\text{\ if\ }\ s\in \mathbb{R}, \   l=1.
\\
\end{cases} 
\end{equation}
Since $$\hat \mu(s, l)=\hat \mu_1(s, l)\hat \nu(s, l), 
\quad s\in  \mathbb{R}, \ l\in \mathbb{Z}(2),$$ then $\mu=\mu_1*\nu$. 
It is clear that $\nu\notin \Gamma(\mathbb{R})*\mathrm{M}^1(\mathbb{Z}(2))$. 
Since (\ref{eq21.06.1}) and (\ref{eq21.06.2}) hold, the above reasoning shows 
that $\nu$ is a decomposable distribution.   \end{proof}

\begin{remark}\label{re2}
  Let a distribution $\mu$ on the group $\mathbb{R}\times \mathbb{Z}(2)$ 
  belong to the semigroup $\Theta$, and let   $\hat\mu(s, l)=f(s, l)$, where the function
   $f(s, l)$ is represented in the form $(\ref{y6})$, and the inequalities 
   $0<\sigma'<\sigma$ and      (\ref{14_12_2}) are fulfilled.
Using Lemma \ref{lem14} and   representation (\ref{14.12.1}) of the characteristic functions of distributions belonging to the semigroup 
$\mathrm{M}^1(\mathbb{Z}(2))$, it is easy to check that
the  distribution  $\mu$ can be represented as $\mu=\lambda*\pi$, where $\lambda\in \Theta$ and $\lambda$ is an indecomposable distribution, and 
$\pi\in\mathrm{M}^1(\mathbb{Z}(2))$. 

Note that any distribution belonging to the semigroup 
$\mathrm{M}^1(\mathbb{Z}(2))$   has no indecomposable factors. 
\end{remark}

Let us now prove the main result of the note. Denote by $\mathrm{I}(\Theta)$ the
class of infinitely divisible distributions, by
$\mathrm{I_0}(\Theta)$ the class of distributions  
which have no indecomposable factors, and by $\mathrm{Ind}(\Theta)$ 
the class of indecomposable distributions in the semigroup $\Theta$.

\begin{theorem}\label{th1}
   The following 
   statements are true:
    
    \renewcommand{\labelenumi}{\arabic{enumi}.}
\begin{enumerate}
  
\item	

 $\mathrm{I}(\Theta)=\Gamma(\mathbb{R})*\mathrm{M}^1(\mathbb{Z}(2));$  

\item

$\mu\in\mathrm{I_0}(\Theta)$ if and only if  $\mu\in\Gamma(\mathbb{R})*\mathrm{M}^1(\mathbb{Z}(2))$ and $\mu$ 
 can not be represented as $\mu=\gamma*m_{\mathbb{Z}(2)}$, where $\gamma$ is a nondegenerate Gaussian distribution on the group $\mathbb{R};$ 

\item

 $\mu\in\mathrm{Ind}(\Theta)$  if and only if the characteristic function
 $\hat\mu(s, l)$ is represented in the form
 $$
\hat\mu(s, l) = \begin{cases}\exp\{-\sigma s^2+i\beta s\}, &\text{\ if\ }\  s\in \mathbb{R}, \ l=0,\\ \kappa\exp\{-\sigma' s^2+i\beta's\}, & \text{\ if\ }\ s\in \mathbb{R}, \   l=1,
\\
\end{cases}
$$ 
and $(\ref{11_12_7})$ is fulfilled.
\end{enumerate}
\end{theorem}
\begin{proof} 1. It is easy to see that all distributions belonging to the 
semigroup $\mathrm{M}^1(\mathbb{Z}(2))$ are infinitely divisible. 
Therefore, if
 $\mu\in\Gamma(\mathbb{R})*\mathrm{M}^1(\mathbb{Z}(2))$, 
 then $\mu$ is an infinitely divisible
 distribution.
Assume that   $\mu\notin\Gamma(\mathbb{R})*\mathrm{M}^1(\mathbb{Z}(2))$. Then 
$\hat\mu(s, l)=f(s, l)$, where the function $f(s, l)$ is represented in the form 
  $(\ref{y6})$ and inequalities  $(\ref{15.3})$ are satisfied.
Suppose that $\mu$ is an infinitely divisible
 distribution.  Then for each natural $n$ there is a distribution $\mu_n\in\Theta$ and an element $(t_n, k_n)\in \mathbb{R}\times \mathbb{Z}(2)$ such that $\mu=\mu_n^{*n}*E_{(t_n, k_n)}$. Hence,
$$
 \hat\mu(s, l)=(\hat\mu_n(s, l))^ne^{it_ns}(-1)^{k_nl}, 
 \quad  s \in\mathbb{R}, \   l\in \mathbb{Z}(2).
$$ This implies that $|\hat\mu(s, l)|=|\hat\mu_n(s, l)|^n$. 
Since $\mu_n\in\Theta$, we have
 $$
|\hat\mu_n(s, l)|= \begin{cases}\exp\{-\frac{\sigma}{n} s^2\}, 
&\text{\ if\ }\  s\in \mathbb{R}, \ l=0,\\ 
|\kappa|^\frac{1}{n}\exp\{-\frac{\sigma'}{n} s^2\}, & 
\text{\ if\ }\ s\in \mathbb{R}, \   l=1,
\\
\end{cases}
 $$
and the inequality $$|\kappa|^\frac{1}{n}\le \sqrt\frac{\sigma'}{\sigma}$$
holds.  Since $n$ is an arbitrary natural number, we get a contradiction.
  Therefore,   $\mu$ is not an infinitely divisible  distribution.

Thus, the class of infinitely divisible distributions in the semigroup $\Theta$ coincides with the class $\Gamma(\mathbb{R})*\mathrm{M}^1(\mathbb{Z}(2))$.

2. Let $\mu\in\Gamma(\mathbb{R})*\mathrm{M}^1(\mathbb{Z}(2))$ and $\mu$ 
is not represented in the form   $\mu=\gamma*m_{\mathbb{Z}(2)}$, where $\gamma$ 
is a nondegenerate Gaussian distribution on the group $\mathbb{R}$. Two cases are possible. 

2a. The characteristic function $\hat\mu(s, l)$ does not vanish.
It follows from the proof of item 1 of Lemma \ref{lem14} and (\ref{14.12.10}) that
 all factors of    $\mu$ also belong to the class 
$\Gamma(\mathbb{R})*\mathrm{M}^1(\mathbb{Z}(2))$, i.e., 
   $\mu$   has no indecomposable factors.
 
 2b. The characteristic function $\hat\mu(s, l)$   vanishes. 
 Then $\mu=E_b*m_{\mathbb{Z}(2)}$, where $b\in\mathbb{R}$. 
 It is obvious that $\mu$   has no indecomposable factors.

 Let $\mu=\gamma*m_{\mathbb{Z}(2)}$, where
 $\gamma$ is a nondegenerate Gaussian distribution on the group $\mathbb{R}$. Then the characteristic function $
\hat\mu(s, l)$  is of the form
$$
\hat\mu(s, l)=\begin{cases}\exp\{-a s^2+ib s\}, 
&\text{\ if\ }\ s\in \mathbb{R}, \ l=0,\\ 0,
&\text{\ if\ }\ s\in \mathbb{R}, \   l=1,
\\
\end{cases}
$$ 
where $a>0$, $b\in \mathbb{R}$.  
It is easy to see that
 $\mu$ has a factor $\mu_1$ such that $\hat\mu_1(s, l)=f(s, l)$, where the function $f(s, l)$ is represented in the form $(\ref{y6})$  and (\ref{11_12_7}) is satisfied. By Lemma \ref{lem14},
  $\mu_1$ is an indecomposable distribution.
  
  If $\mu\notin\Gamma(\mathbb{R})*\mathrm{M}^1(\mathbb{Z}(2))$, then by Corollary \ref{co2},
   $\mu$ has an indecomposable factor.

3. Obviously, all nondegenerate distributions belonging to the class $\Gamma(\mathbb{R})*\mathrm{M}^1(\mathbb{Z}(2))$ are decomposable. Let $\mu$ be a nondegenerate distribution  and let 
$\mu\notin\Gamma(\mathbb{R})*\mathrm{M}^1(\mathbb{Z}(2))$. Then  
$\hat\mu(s, l)=f(s, l)$, where the function $f(s, l)$ is represented in the form 
  $(\ref{y6})$ and  inequalities $(\ref{15.3})$ are satisfied.
 It follows from (\ref{15.3}) that either (\ref{11_12_7}) is true or  
  inequalities (\ref{14_12_2}) hold. As follows from the proof of Proposition \ref{pr3}, if   inequalities (\ref{14_12_2}) hold, then the distribution $\mu$ is decomposable. Thus, the validity of (\ref{11_12_7}), according to Lemma \ref{lem14}, is not only  a sufficient condition for the distribution $\mu$ to be indecomposable, but also a necessary one.
\end{proof}
 
\textbf{Acknowledgements}
This article was written during my stay at 
the Department of Mathematics University of Toronto as a Visiting Professor. 
I am very grateful to Ilia Binder for his invitation and support. 
I would like also to thank Alexander Il'inskii for some useful  remarks.

\vskip 1 cm

\noindent B. Verkin Institute for Low Temperature Physics and Engineering\\ 
of the National Academy of Sciences of Ukraine, \\ 
47 Nauky Ave., Kharkiv, 61103, Ukraine

\medskip

\noindent Department of Mathematics  
University of Toronto \\
40 St. George Street
Toronto, ON,  M5S 2E4
Canada 

\medskip

\noindent E-mail: {feldman@ilt.kharkov.ua}, gennadiy\_f@yahoo.co.uk


\begin{thebibliography}{99}

\bibitem{F_solenoid} G.M. Feldman, \textit{On a characterization theorem for connected locally compact Abelian groups}, J. Fourier Anal. Appl. \textbf{ 26}, Paper No. 14  (2020), 22 pp. 

\bibitem{POTA}  G.M. Feldman, \textit{On a characterization  theorem for locally compact Abelian groups containing  an element  of order $2$},  
Potential Analysis. \textbf{56}   (2022), 297--315. 

\bibitem{book2023} G. Feldman,  \textit{Characterization of probability distributions on locally compact Abelian groups}, Mathematical Surveys and Monographs, 
{\bf 273}. American Mathematical Society, Providence, RI, 2023.

\bibitem{Il} A.I. Il'inskii, \textit{Some remarks on the uniqueness of the 
extension of measures from subsets of the group  $\mathbf{Z}_2\times  \mathbf{R}$},  
 J. Soviet Math. \textbf{57}, no. 4    (1991), 3242--3245. 
  
\bibitem{KaLiRa}  A. M. Kagan,  Yu. V. Linnik,  C.R. Rao,
 \textit{Characterization problems in mathematical statistics},
Wiley Series in Probability and Mathematical
Statistics. John Wiley $\&$ Sons, New York, 1973. 

\bibitem{O1} I.V. Ostrovskii, \textit{A description of the class $I_0$ in a special semigroup of probability measures}, Dokl. Akad. Nauk SSSR, \textbf{209} (1973), 788--791 (Russian); Engl. transl.: Soviet Math. Dokl. \textbf{14} (1973), 525--529.

\bibitem{Pa} K.R. Parthasarathy,  \textit{Probability measures
on metric spaces}, Academic Press, New York and London, 1967.

\bibitem{Tr1} I.P. Trukhina, \textit{A problem connected with the arithmetic of probability measures on spheres}, Zap. Nauchn. Sem. Leningrad. Otdel. Mat. Inst. Steklov. (LOMI), \textbf{87} (1979), 143--158 (Russian).

\bibitem{Tr2} I.P. Trukhina, \textit{The arithmetic of spherically symmetric 
measures in Lobachevskij space}, Teor. Funkts.  Funktsional. Anal. 
i Prilozhen, no {34} (1980), 136--146 (Russian).

\bibitem{Tr3}  I.P. Trukhina, \textit{A note on stability of the Cram\'er theorem on the group  $\mathbf{R}\times\mathbf{Z}_2$},  J. Soviet Math. 
\textbf{47}, no. 5 (1989), 2761--2765.

\bibitem{Zo} V.M. Zolotarev, \textit{On a general theory of multiplication of independent random variables}, Dokl. Akad. Nauk SSSR, \textbf{142} (1962), 788--791 (Russian); Engl. transl.: Soviet Math. Dokl. \textbf{3} (1962), 166--170.

\end{thebibliography}
\end{document}